\documentclass[oneside]{amsart}
\usepackage{amssymb,amsmath}
\usepackage{pgfplots}
\pgfplotsset{compat=1.5}

\def \R{\mathbb{R}}
\def \N{\mathbb{N}}
\def \Q{\mathbb{Q}}

\theoremstyle{theorem}
\newtheorem{theorem}{Theorem}[section]
\newtheorem{corollary}[theorem]{Corollary}

\theoremstyle{definition}
\newtheorem{definition}[theorem]{Definition}

\begin{document}

\title{Early Transcendental Analysis}
\markright{Early Transcendental Analysis}
\author{Simon Cowell and Philippe Poulin}

\maketitle

\begin{abstract}
In \cite{Weintraub_paper} Steven Weintraub presents a rigorous justification of the ``early transcendental" calculus textbook approach to the exponential and logarithmic functions. However, he uses tools such as term-by-term differentiation of infinite series. We present a rigorous treatment of the early transcendental approach suitable for a first course in analysis, using mainly the supremum property of the real numbers.
\end{abstract}

\section{Motivation}

As Weintraub states in his 1997 article in the American Mathematical Monthly \cite{Weintraub_paper}, several current calculus texts have ``early transcendental" versions, in which the exponential and logarithmic functions are introduced early in the text. The existence and properties of those functions are justified by ``hand-waving" arguments.

In contrast to this, ``late transcendental" versions of the calculus texts have more rigorous proofs. A disadvantage of the late transcendental approach is that relatively advanced topics such as integration or infinite series are introduced before the more elementary concepts of the exponential and logarithmic functions.
However, the level of rigor appropriate to a calculus course is sufficient for a convincing ``late transcendental" treatment, but apparently not for a self-contained ``early transcendental" one.

Teachers who prefer the early transcendental approach (ETA) over the late transcendental approach (LTA) can be confident that the ETA is no less rigorous than the LTA. An advantage of the LTA is that it is close to the historical development \cite{Edwards_book}. An advantage of the ETA is that it is very intuitive. We show here a way to justify the ETA rigorously, in a manner suitable for a first (post-calculus sequence) course in Analysis.

Arguably, the r\^ole of rigor in a calculus course is to introduce the student to the idea of mathematical proof rather than to endow our statements with unshakeable certainty. The latter goal is certainly unattainable in a traditional calculus course, in which the real numbers themselves are defined only intuitively - they are not constructed. Therefore the teacher choosing the ETA over the LTA need not worry that the ETA seems less self-contained than the LTA. Indeed, the justification of the ETA can simply be postponed to an Analysis course.

The essence of the ETA, as it is currently presented in calculus textbooks, is as follows:
For $a = 1/2$, plot a sequence of points with integer coordinates $(n, a^n)$ and connect them with a smooth curve, as in figure a) below.
Repeat the experiment with $a = 2$ and $a = 5$ as in figures b) and c).
Observe that the curves all seem to have tangent lines at the point $(0, 1)$ and sketch those tangent lines, as in figures a), b) and c). Notice that the slope of the tangent line seems to increase strictly and smoothly with $a$, without upper or lower bound.
Conjecture that there is therefore a unique value of $a$ for which the tangent line has slope equal to one, and define $e$ as this value. Sketch the corresponding curve and tangent line as in figure d) below.
In general, define $\ln a$ as the slope of the tangent line to $y = a^x$ at the point $(0, 1)$.
Thus $\ln e = 1$ by definition of $e$.

\begin{tikzpicture}
\begin{axis}[
title={\bf{a)} $y=\left (\frac{1}{2} \right )^x$},
xlabel={$x$},
ylabel={$y$},
xmin=-3, xmax=3,
ymin=-2.3, ymax=27,
x post scale = .65,
y post scale = .65,
]
\addplot[
	black,
	domain=-3:3,
	samples = 50,
]
	{(1/2)^x};
\addplot[
	only marks,
	mark size = 1.5pt,
	color=black,	
]
table{
	x	y
	-3	8
	-2	4
	-1	2
	0	1
	1	.5
	2	.25
	3	.125
};
\addplot[
	black,
	dashed,
	domain=-3:3,
	samples = 2,
]
	{ln(1/2)*x+1};
\end{axis}
\end{tikzpicture}
\begin{tikzpicture}
\begin{axis}[
title={\bf{b)} $y=2^x$},
xlabel={$x$},
ylabel={$y$},
xmin=-3, xmax=3,
ymin=-2.3, ymax=27,
x post scale = .65,
y post scale = .65,
]
\addplot[
	black,
	domain=-3:3,
	samples = 50,
]
	{2^x};
\addplot[
	only marks,
	mark size = 1.5pt,
	color=black,	
]
table{
	x	y
	-3	.125
	-2	.25
	-1	.5
	0	1
	1	2
	2	4
	3	8
};
\addplot[
	black,
	dashed,
	domain=-3:3,
	samples = 2,
]
	{ln(2)*x+1};
\end{axis}
\end{tikzpicture}

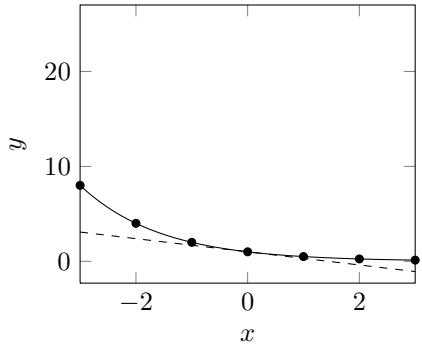
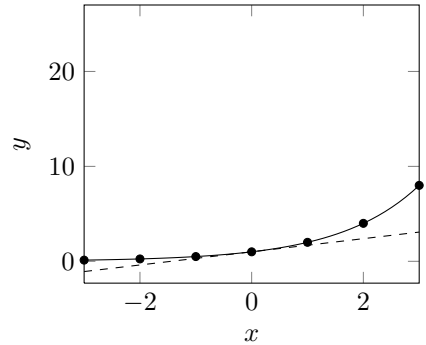
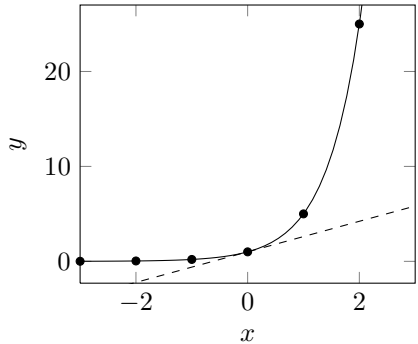
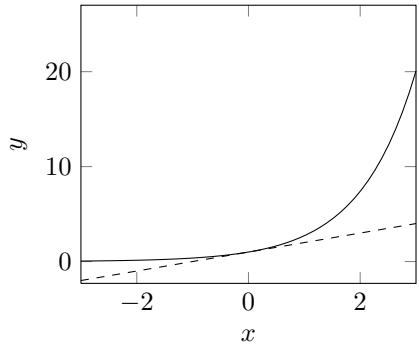
\begin{tikzpicture}
\begin{axis}[
title={\bf{c)} $y=5^x$},
xlabel={$x$},
ylabel={$y$},
xmin=-3, xmax=3,
ymin=-2.3, ymax=27,
x post scale = .65,
y post scale = .65,
]
\addplot[
	black,
	domain=-3:3,
	samples = 50,
]
	{5^x};
\addplot[
	only marks,
	mark size = 1.5pt,
	color=black,	
]
table{
	x	y
	-3	.008
	-2	.04
	-1	.2
	0	1
	1	5
	2	25
};
\addplot[
	black,
	dashed,
	domain=-3:3,
	samples = 2,
]
	{ln(5)*x+1};
\end{axis}
\end{tikzpicture}
\begin{tikzpicture}
\begin{axis}[
title={\bf{d)} $y=e^x$},
xlabel={$x$},
ylabel={$y$},
xmin=-3, xmax=3,
ymin=-2.3, ymax=27,
x post scale = .65,
y post scale = .65,
]
\addplot[
	black,
	domain=-3:3,
	samples = 50,
]
	{e^x};
\addplot[
	black,
	dashed,
	domain=-3:3,
	samples = 2,
]
	{x+1};
\end{axis}
\end{tikzpicture}

\noindent
Note that by associativity of multiplication
\[
a^{x+y} = \overbrace{(a \cdot a \cdot \ldots \cdot a)}^{x + y \ \mathrm{factors}} = \overbrace{(a \cdot a \cdot \ldots \cdot a)}^{x \ \mathrm{factors}} \overbrace{(a \cdot a \cdot \ldots \cdot a)}^{y \ \mathrm{factors}} = a^x a^y
\]
for positive integer values of $x$ and $y$.
Conjecture that the relationship $a^{x+y} = a^x a^y$ extends to arbitrary real values of $x$ and $y$.
Using this conjecture, compute
\[
\frac{a^{x+h}-a^x}{h} = \frac{a^x a^h - a^x}{h} = a^x \frac{a^h - 1}{h}.
\]
Take limits as $h \to 0$, yielding
\[
\frac{d}{dx} a^x = (\ln a) a^x.
\]
In particular, note that
\[
\frac{d}{dx} e^x = e^x.
\]
This concludes our brief description of the heart of the ETA, at the level of a Calculus course.

The author of \cite{Weintraub_paper} essentially defines the exponential function as an infinite series, and differentiates it using term-by-term differentiation of the series. Despite its historical precedent, \cite{Edwards_book}, we feel that this is a departure from the spirit of the ETA. Alternatively, the exponential function $a^x$ can be introduced as the continuous extension from the rationals to the reals of the familiar function $a^{p/q}$, as in \cite{Tao_book}. We prove, using this definition and using only elementary techniques such as the inequality of arithmetic and geometric means, that $a^x$ is differentiable. We retrace exactly the steps of the ETA as described above, this time with complete rigor. We recommend that, if a calculus class is taught using the (nonrigorous) ETA, and if it is followed by a class in Analysis, then that Analysis class should use the rigorous ETA as explained herein.

\section{Literature review}

We refer the interested reader to \cite{Edwards_book} for a fascinating history of the logarithmic and exponential functions.

The authors of 
\cite{Bartle_and_Sherbert_book},
\cite{Burkill_book},
\cite{Burkill_and_Burkill_book},
\cite{Burrill_and_Knudsen_book},
\cite{Cooper_book},
\cite{Gaskill_and_Narayanaswami_book},
\cite{Kirkwood_book},
\cite{Krantz_book},
\cite{Leadership_Project_Committee_book},
\cite{Maude_book},
\cite{Reade_book},
\cite{Rudin_book},
\cite{Stromberg_book},
\cite{Whitaker_and_Watson_book}
and
\cite{White_book}
define $e^x$ by its Taylor-McLaurin series.
Of these,
\cite{Burkill_and_Burkill_book},
\cite{Burrill_and_Knudsen_book},
\cite{Cooper_book},
\cite{Krantz_book},
\cite{Leadership_Project_Committee_book},
\cite{Maude_book},
\cite{Reade_book}
and
\cite{White_book}
then prove that $\frac{d}{dx} e^x = e^x$ by term-by-term differentiation of the Taylor-McLaurin series.
The authors of
\cite{Bartle_and_Sherbert_book},
\cite{Burkill_book},
\cite{Gaskill_and_Narayanaswami_book},
\cite{Kirkwood_book},
\cite{Rudin_book},
\cite{Stromberg_book}
and
\cite{Whitaker_and_Watson_book}
on the other hand prove that $\frac{d}{dx} e^x = e^x$ as follows: First by using the series definition of $e^x$ to prove that $e^{x+y} = e^x e^y$. Then by manipulating the series and using a theorem about continuity of power series, it follows that the required limit exists.

The authors of
\cite{Courant_book_I},
\cite{Courant_and_John_book},
\cite{Fulks_book},
\cite{Gaughan_book_4},
\cite{Gaughan_book_5},
\cite{Protter_and_Morrey_book_91}
and
\cite{Smith_book}
define $\ln x$ by a definite integral. They then define $e^x$ as the inverse of $\ln x$, except for \cite{Courant_and_John_book} who define $a^x$ in the same way as we do, and define $e$ by $\sum_{n=0}^{\infty} \frac{1}{n!}$, and then prove that $e^x$ is the inverse of $\ln x$.
In all eight of these sources, $e^x$ is then differentiated using the chain rule or the Inverse Function Theorem, and the Fundamental Theorem of Calculus.

In
\cite{Granville_and_Smith_and_Longley_book},
\cite{Haeussler_and_Paul_book}
and
\cite{Kaplan_book}
$e$ is defined as $\lim_{x \to 0} (1 + x)^{1/x}$.
In \cite{Haeussler_and_Paul_book} and \cite{Kaplan_book} this limit is also used to give a numerical approximation to the value of $e$.
These three sources define $\ln x$ as the inverse function of $e^x$.
They use the existence of the above limit and properties of logarithms to prove that $\frac{d}{dx} \ln x = \frac{1}{x}$.
They then deduce, via the chain rule, that $\frac{d}{dx} e^x = e^x$.
The authors of \cite{Haeussler_and_Paul_book} draw a graph to justify the existence of the key limit.
There is no graph in \cite{Kaplan_book}, nor are there as many details as in \cite{Haeussler_and_Paul_book}.
The authors of \cite{Granville_and_Smith_and_Longley_book} admit that proving that $\lim_{x \to 0} \left (1 + \frac{1}{x} \right )^x$ exists is beyond the scope of their book.

The authors of \cite{Gupta_and_Rani_book} define $a^x$ as we do, and they define $\log_a x$ as the inverse of $a^x$. They define $e$ as $\lim_{n \to \infty} \left ( 1 + \frac{1}{n} \right )^n$.

Reference
\cite{Goldstein_and_Lay_and_Schneider_book}
takes the same approach as we do but with much less rigour. They claim without proof that $a^x$ is differentiable at 0, and that there is a unique value $e$ defined as we define it here.

In
\cite{Agnew_book}
the author defines $a^x$ as we do, then shows that its derivative at zero exists by writing it in terms of a telescoping sum which becomes a Riemann sum, which converges to a certain definite integral.

In
\cite{Goldberg_book}
$\sinh^{-1} x$ is defined by a definite integral, and $\cosh x$ as $\sqrt{1  + \sinh^2 x}$.
The author then uses the Inverse Function Theorem and the Fundamental Theorem of Calculus to differentiate $\sinh x$, showing that $\sinh' x = \sqrt{1 + \sinh^2 x}$ hence $\sinh'x = \cosh x$.
It then follows, from the equation before last, that $\sinh'' x = \sinh x$, hence $\cosh' x = \sinh x$.
The author then defines $e^x$ as $\cosh x + \sinh x$, and deduces from the above that $\frac{d}{dx} e^x = e^x$.
\footnote{
An analagous approach is used in \cite{Goldberg_book} to treat those transcendental functions which we otherwise omit to mention in the present work, namely the trigonometric functions. The author defines $\sin x$ and $\cos x$ and differentiates them.
He points out that the usual geometric proof of the existence of the limit required to differentiate $\sin x$ depends on knowing the area of a circle, which is usually found by computing an integral using a trigonometric substitution, creating a circular argument. His method avoids this circularity.}

\vspace{.5in}

\section{Theory}

That rational powers $a^q$ of positive reals $a$ exist is a consequence of the supremum property of $\R$. See for example \cite{Tao_book}. By a further application of the supremum property, one can define arbitrary real powers of positive reals as follows:

\begin{definition}
Given $a > 0$ and $x \in \R$, let $a^x = \sup \{a^q : q \in \Q, q < x\}$.
\end{definition}

We take for granted the proofs of the following two theorems, see for instance \cite{Tao_book}.
\begin{theorem}
\label{thm: basic properties of x^a}
For $x \in \R$ with $x > 0$ and $a \in \R$ \\
(i) $x^a$ is strictly decreasing in $x > 0$ for $a < 0$ \\
(ii) $x^a$ is strictly increasing in $x > 0$ for $a > 0$ \\
(iii) $x^a$ is continuous in $x > 0$
\end{theorem}

\begin{theorem}
\label{thm: basic properties of a^x}
For $a, b \in \R$ with $a, b > 0$ and $x, y \in \R$ \\
(i) $a^{x + y} = a^x a^y$ \\
(ii) $(a^x)^y = a^{xy}$ \\
(iii) For $0 < a < 1$, $a^x$ is strictly decreasing in $x$ and for $a > 1$, $a^x$ is strictly increasing in $x$ \\
(iv) $a^x$ is continuous in $x$ \\
(v) $a^1 = a$ \\
(vi) $a^0 = 1$ \\
(vii) $a^x > 0$ \\
(viii) For $0<a<1$, $\lim_{x \to -\infty} a^x = + \infty$ and $\lim_{x \to + \infty} a^x = 0$, and  for $a > 1$, $\lim_{x \to -\infty} a^x = 0$ and $\lim_{x \to +\infty} a^x = +\infty$ \\
(ix) $a^x b^x = (ab)^x$.
\end{theorem}

\begin{theorem}[The Arithmetic-Geometric Mean Inequality]
\label{thm: arithmetic-geometric mean inequality}
For non-negative real numbers $a_1, a_2, \ldots, a_n$,
\[
\sqrt[n]{a_1 a_2 \cdots a_n} \leq \frac{a_1 + a_2 + \cdots + a_n}{n}
\]
with equality if and only if $a_1 = a_2 = \ldots = a_n$.
\end{theorem}
\begin{proof}
The statement is trivially true when $n = 1$.
In case $n = 2$, given $a_1, a_2 \geq 0$ we have
$0 \leq (\sqrt{a_1} - \sqrt{a_2})^2 = a_1 - 2 \sqrt{a_1 a_2} + a_2$, whence $2 \sqrt{a_1 a_2} \leq a_1 + a_2$, and $\sqrt{a_1a_2} \leq \frac{a_1 + a_2}{2}$. Evidently equality obtains if and only if $a_1 = a_2$.
We will prove the general statement by induction on $n$.
Indeed, suppose the statement holds for some $n \in \N^*=\N \setminus \{0\}$ and suppose that $a_1, \ldots, a_n, a_{n+1} \geq 0$.
Let $A = \frac{a_1 + \cdots + a_{n+1}}{n+1}$.
By the case $n=2$ and by the induction hypothesis and by Theorem \ref{thm: basic properties of a^x} parts (ii) and (ix) we have
\begin{align*}
\sqrt[2n]{a_1 \cdots a_{n+1} A^{n-1}}
&= \sqrt{\sqrt[n]{a_1 \cdots a_n} \sqrt[n]{a_{n+1} A^{n-1}}} \\
&\leq \frac{\sqrt[n]{a_1 \cdots a_n} + \sqrt[n]{a_{n+1} A^{n-1}}}{2} \\
&\leq \frac{a_1 + \cdots + a_n}{2n} + \frac{a_{n+1} + (n-1) A}{2n} \\
&= \frac{(n+1)A + (n-1)A}{2n} \\
&= A,
\end{align*}
the condition for equality being that $\sqrt[n]{a_1 \cdots a_n} = \sqrt[n]{a_{n+1} A^{n-1}}$ and $a_1 = \cdots = a_n$ and $a_{n+1} = A$, equivalently that $a_1 = \cdots = a_n = a_{n+1}$. Dividing both sides of the inequality $\sqrt[2n]{a_1 \cdots a_{n+1} A^{n-1}} \leq A$ by $A^{\frac{n-1}{2n}}$ we have $\sqrt[2n]{a_1 \cdots a_{n+1}} \leq A^{\frac{n+1}{2n}}$, and raising both sides to the power of $\frac{2n}{n+1}$ gives $\sqrt[n+1]{a_1 \cdots a_{n+1}} \leq A$, as required.
This completes the proof by induction.
\end{proof}

\begin{definition}
Denote $\{a \in \R : a>0\}$ by $\R^+$.
\end{definition}

\begin{theorem}
\label{thm: a^x has nondecreasing difference quotient}
For $a \in \R^+$ and $h, k \in \R \setminus \{0\}$ with $h < k$,
\[
\frac{a^h - 1}{h} \leq \frac{a^k - 1}{k}.
\]
\end{theorem}
\begin{proof}
Let $a \in \R^+$ and let $h, k \in \Q$ with $0 < h < k$.
Then there exist $m, n \in \N^*$ such that $\frac{h}{k} = \frac{m}{m+n}$.
By Theorem \ref{thm: arithmetic-geometric mean inequality} we have $\sqrt[m+n]{(a^k)^m \cdot 1 \cdot 1 \cdots 1} \leq \frac{m a^k + 1 + 1 + \cdots + 1}{m + n}$ i.e. $\sqrt[m+n]{(a^k)^m} \leq \frac{m a^k + n}{m+n}$. By Theorem \ref{thm: basic properties of a^x} part (ii) this means that $a^{k \frac{m}{m+n}} \leq \frac{m}{m + n} a^k + 1 - \frac{m}{m + n}$, i.e. $a^h \leq \frac{h}{k} a^k + 1 - \frac{h}{k}$.
Given that $h > 0$ this yields
\begin{equation}
\label{eqn: difference quotient inequality for 0 < h < k, h, k rational}
\frac{a^h - 1}{h} \leq \frac{a^k - 1}{k}, \quad h, k \in \Q, \ 0 < h < k.
\end{equation}
Now assume that $h, k \in \Q$ with $h < k < 0$.
From \eqref{eqn: difference quotient inequality for 0 < h < k, h, k rational} we have $\frac{a^{-k} - 1}{-k} \leq \frac{a^{-h} - 1}{-h}$. Replacing $a$ by $a^{-1}$ gives $\frac{(a^{-1})^{-k} - 1}{-k} \leq \frac{(a^{-1})^{-h} - 1}{-h}$, and by Theorem \ref{thm: basic properties of a^x} part (ii) we have $\frac{a^k - 1}{-k} \leq \frac{a^h - 1}{-h}$, hence
\begin{equation}
\label{eqn: difference quotient inequality for h < k < 0, h, k rational}
\frac{a^h - 1}{h} \leq \frac{a^k - 1}{k}, \quad h, k \in \Q, \ h < k < 0.
\end{equation}
Now let $h \in \R$ with $h>0$. Then $(a^h - 1)^2 \geq 0$, $0 \leq (a^h)^2 - 2 a^h + 1$ and $a^h - 1 \leq (a^h)^2 - a^h$, so by Theorem \ref{thm: basic properties of a^x} parts (i), (vi) and (vii),
$1 - a^{-h} \leq a^h - 1$ and
\begin{equation}
\label{eqn: difference quotient inequality for -h and h, h real}
\frac{a^{-h} - 1}{-h} \leq \frac{a^h - 1}{h}, \quad h \in \R, \ h > 0.
\end{equation}
Let $h, k \in \Q$ with $h < 0 < k$, and let $x = \min(|h|, |k|)$. Then by inequalities \eqref{eqn: difference quotient inequality for h < k < 0, h, k rational}, \eqref{eqn: difference quotient inequality for -h and h, h real} and \eqref{eqn: difference quotient inequality for 0 < h < k, h, k rational},
$\frac{a^h - 1}{h} \leq \frac{a^{-x} - 1}{-x} \leq \frac{a^x - 1}{x} \leq \frac{a^k - 1}{k}$, hence
\begin{equation}
\label{eqn: difference quotient inequality for h < 0 < k, h, k rational}
\frac{a^h - 1}{h} \leq \frac{a^k - 1}{k}, \quad h, k \in \Q, \ h < 0 < k.
\end{equation}
Considering inequalities \eqref{eqn: difference quotient inequality for 0 < h < k, h, k rational} , \eqref{eqn: difference quotient inequality for h < k < 0, h, k rational} and \eqref{eqn: difference quotient inequality for h < 0 < k, h, k rational} we have
\begin{equation}
\label{eqn: difference quotient inequality for nonzero rational h and k}
\frac{a^h - 1}{h} \leq \frac{a^k - 1}{k}, \quad h, k \in \Q \setminus \{0\}, \ h < k.
\end{equation}
It remains to extend this inequality from rational $h$ and $k$ to real $h$ and $k$.
Indeed, let $h, k \in \R \setminus \{0\}$ with $h < k$.
Choose, as we may by the density of the rationals in the reals, sequences $(h_n) \subset \Q \setminus \{0\}$ and $(k_n) \subset \Q \setminus \{0\}$ such that $\lim_{n \to \infty} h_n = h$ and $\lim_{n \to \infty} k_n = k$.
Then for $n$ sufficiently large $h_n < k_n$, hence $\frac{a^{h_n} - 1}{h_n} \leq \frac{a^{k_n} - 1}{k_n}$ by \eqref{eqn: difference quotient inequality for nonzero rational h and k}.
Taking limits as $n$ tends to infinity and using Theorem \ref{thm: basic properties of a^x} part (iv) we have $\frac{a^h - 1}{h} \leq \frac{a^k - 1}{k}$, which completes the proof.
\end{proof}

\paragraph{Remark} Theorem \ref{thm: a^x has nondecreasing difference quotient} is equivalent to the statement that $a^x$ is a convex function of $x$ provided that $a$ is a positive real number.

\begin{theorem}
\label{thm: a^x is differentiable at zero}
For $a \in \R^+$, the function $x \mapsto a^x$ is differentiable at $0$.
\end{theorem}
\begin{proof}
Let $a \in \R^+$.
By Theorem \ref{thm: basic properties of a^x} part (vi) it will suffice to show that $\lim_{h \to 0} \frac{a^h - 1}{h}$ exists.
In fact, by Theorem \ref{thm: a^x has nondecreasing difference quotient} the nonempty set $\{\frac{a^h - 1}{h} : h \in \R, h > 0\}$ is bounded below, for instance by $\frac{a^{-1} - 1}{-1}$. Therefore $\inf_{h \in \R, h > 0} \frac{a^h - 1}{h}$ exists and is finite. Again by Theorem \ref{thm: a^x has nondecreasing difference quotient}, $\lim_{h \to 0^+} \frac{a^h - 1}{h}$ exists and equals $\inf_{h \in \R, h > 0} \frac{a^h - 1}{h}$.
We have
\begin{align*}
\lim_{h \to 0^+} \frac{a^h - 1}{h}
&= \lim_{h \to 0^-} \frac{a^{-h} - 1}{-h} \\
&= a^0 \lim_{h \to 0^-} \frac{a^{-h} - 1}{-h} \\
&= \left ( \lim_{h \to 0^-} a^h \right ) \left ( \lim_{h \to 0^-} \frac{a^{-h} - 1}{-h} \right ) \\
&= \lim_{h \to 0^-} \left (a^h \frac{a^{-h} - 1}{-h} \right ) \\
&= \lim_{h \to 0^-} \frac{a^h - 1}{h},
\end{align*}
where we have used Theorem \ref{thm: basic properties of a^x} parts (i), (iv) and (vi).
Therefore $\lim_{h \to 0^-} \frac{a^h - 1}{h}$ exists and equals $\lim_{h \to 0^+} \frac{a^h - 1}{h}$, hence $\lim_{h \to 0} \frac{a^h - 1}{h}$ exists, as required.
\end{proof}

\begin{definition}
Define the function $\ln: \R^+ \to \R$ by $\ln a = \frac{d}{dx} a^x|_{x=0}$.
\end{definition}

\begin{theorem}
\label{thm: a^x is continuously differentiable everywhere}
For $a > 0$, $a^x$ is continuously differentiable in $x$ at every $x \in \R$. Moreover for all $x \in \R$,
\begin{equation}
\label{eqn: the derivative of a^x is ln a times a^x}
\frac{d}{dx} a^x = (\ln a) a^x.
\end{equation}
\end{theorem}
\begin{proof}
Let $a > 0$ and let $x_0 \in \R$.
We have
\[
a^{x_0} \lim_{h \to 0} \frac{a^h - 1}{h} = \lim_{h \to 0} \frac{a^{x_0+h} - a^{x_0}}{h}
\]
by Theorem \ref{thm: basic properties of a^x} part (i).
By Theorem \ref{thm: a^x is differentiable at zero} the limit on the left hand side exists and equals $\ln a$. Since $x_0$ is arbitrary, $\frac{d}{dx} a^x$ exists everywhere and obeys equation \eqref{eqn: the derivative of a^x is ln a times a^x}.
The continuity of the derivative now follows from equation \eqref{eqn: the derivative of a^x is ln a times a^x} and Theorem \ref{thm: basic properties of a^x} part (iv).
\end{proof}

\begin{theorem}
\label{thm: the slope of a^x at x = 0 varies continuously and increasingly with a > 0}
For $a > 0$, $\ln a$ varies continuously with $a$ and is nondecreasing in $a$.
\end{theorem}
\begin{proof}
By Theorem \ref{thm: basic properties of x^a} parts (i) and (ii) $a^h$ is strictly decreasing in $a > 0$ given a fixed $h < 0$, and $a^h$ is strictly increasing in $a > 0$ given a fixed $h > 0$. It follows that, for a fixed $h \in \R \setminus \{0\}$, $\frac{a^h - 1}{h}$ is strictly increasing in $a > 0$. Therefore $\ln a = \lim_{h \to 0} \frac{a^h - 1}{h}$ is nondecreasing in $a > 0$, as required.
Fixing $a_0 > 0$ it follows that $\lim_{a \to a_0^-} \ln a$ exists and equals $\sup_{0 < a < a_0} \ln a$ and that $\lim_{a \to a_0^+} \ln a$ exists and equals $\inf_{a > a_0} \ln a$.
However, by Theorem \ref{thm: a^x has nondecreasing difference quotient} we have $\ln a = \sup_{h < 0} \frac{a^h - 1}{h}$. Therefore
\begin{align*}
\lim_{a \to a_0^-} \ln a	&= \sup_{0 < a < a_0} \ln a \\														&= \sup_{0 < a < a_0} \sup_{h < 0} \frac{a^h - 1}{h} \\
						&= \sup_{h < 0} \sup_{0 < a < a_0} \frac{a^h - 1}{h} \\
						&= \sup_{h < 0} \lim_{a \to a_0^-} \frac{a^h - 1}{h} \\
						&= \sup_{h < 0} \frac{a_0^h - 1}{h} \\
						&= \ln a_0,
\end{align*}
where we have also used Theorem \ref{thm: basic properties of x^a} part (iii).
Similarly, we show that $\lim_{a \to a_0^+} \ln a = \ln a_0$.
Indeed, by Theorem \ref{thm: a^x has nondecreasing difference quotient} we have $\ln a = \inf_{h > 0} \frac{a^h - 1}{h}$. Therefore
\begin{align*}
\lim_{a \to a_0^+} \ln a	&= \inf_{a > a_0} \ln a \\														&= \inf_{a > a_0} \inf_{h > 0} \frac{a^h - 1}{h} \\
						&= \inf_{h > 0} \inf_{a > a_0} \frac{a^h - 1}{h} \\
						&= \inf_{h > 0} \lim_{a \to a_0^+} \frac{a^h - 1}{h} \\
						&= \inf_{h > 0} \frac{a_0^h - 1}{h} \\
						&= \ln a_0,
\end{align*}
where we have used again Theorem \ref{thm: basic properties of x^a} part (iii).
Therefore $\lim_{a \to a_0} \ln a$ exists and equals $\ln a_0$, so $\ln a$ is continuous at $a_0$. Since $a_0 > 0$ is arbitrary, the Theorem is proved.
\end{proof}

\begin{theorem}
\label{thm: the existence of e}
There is a unique real number $e > 0$ such that $\ln e = 1$. Moreover $2 < e < 3$.
\end{theorem}
\begin{proof}
By Theorem \ref{thm: a^x has nondecreasing difference quotient} we have
\begin{align*}
\ln 2 = \inf_{h > 0} \frac{2^h - 1}{h}	&\leq \frac{2^{\frac{1}{2}} - 1}{\frac{1}{2}} \\
									&= 2(\sqrt{2} - 1) \\
									&< 1
\end{align*}
and
\begin{align*}
\ln 3 = \sup_{h < 0} \frac{3^h - 1}{h}	&\geq \frac{3^{-\frac{1}{6}} - 1}{-\frac{1}{6}} \\
									&= 6 \left ( 1 - \frac{1}{\sqrt[6]{3}} \right ) \\
									&> 1.
\end{align*}
By the Intermediate Value Theorem and Theorem \ref{thm: the slope of a^x at x = 0 varies continuously and increasingly with a > 0} there exists some real number $e \in (2,3)$ such that $\ln e = 1$.
Now suppose that there is some $b > 0$ with the property that $\ln b = 1$. Then by the quotient rule and equation \eqref{eqn: the derivative of a^x is ln a times a^x}, $\frac{d}{dx} \frac{e^x}{b^x} = \frac{(\ln e) e^x b^x - e^x (\ln b) b^x}{(b^x)^2} = \frac{e^x b^x - e^x b^x}{(b^x)^2} = 0$.
Therefore $e^x = C b^x$ for some $C \in \R$.
Setting $x = 0$ and applying Theorem \ref{thm: basic properties of a^x} part (vi) shows that $C = 1$, hence $e^x = b^x$ for all $x \in \R$ and in particular by Theorem \ref{thm: basic properties of a^x} part (v) $e = e^1 = b^1 = b$.
Therefore $e$ is unique.
\end{proof}

By equation \eqref{eqn: the derivative of a^x is ln a times a^x} and Theorem \ref{thm: the existence of e} we have the following.
\begin{corollary}
\label{cor: e^x is its own derivative}
For $x \in \R$, $\displaystyle \frac{d}{dx} e^x = e^x$.
\end{corollary}

By Theorem \ref{thm: basic properties of a^x} parts (iii), (vii) and (viii), $a^x$ is strictly monotone and maps $\R$ onto $\R^+$, provided $a \in \R^+ \setminus \{1\}$. Therefore for such $a$, the function $a^x : \R \to \R^+$ is invertible.
\begin{definition}
\label{def: logarithm to base a}
For $a \in \R^+ \setminus \{1\}$, $x \in \R^+$ and $y \in \R$ we define the logarithm to base $a$, $\log_a : \R^+ \to \R$ by $\log_a x = y$ if and only if $a^y = x$.
\end{definition}

\begin{theorem}
\label{thm: our ln is the logarithm to base e}
For $a > 0$ \ $\ln a = \log_e a$.
\end{theorem}
\begin{proof}
Let $a > 0$. By Theorem \ref{thm: basic properties of a^x} part (ii), $a^x = (e^{\log_e a})^x = e^{x \log_e a}$. By equation \eqref{eqn: the derivative of a^x is ln a times a^x}, Theorem \ref{thm: the existence of e} and the Chain Rule we have
\[
(\ln a) a^x = \frac{d}{dx} a^x = \frac{d}{dx} e^{x \log_e a} = e^{x \log_e a} \frac{d}{dx} (x \log_e a) = (\log_e a) a^x
\]
and the result follows on dividing both sides by $a^x$.
\end{proof}

The function $\ln: \R^+ \to \R$ is called the Natural logarithm. It is sometimes also called the Napierian logarithm, although this is misleading - Napier's logarithm tables were not in fact tables of logarithms to base $e$ \cite{Edwards_book}.
In view of Theorem \ref{thm: basic properties of a^x} part (iii) the conclusion of Theorem \ref{thm: the slope of a^x at x = 0 varies continuously and increasingly with a > 0} can be strengthened:

\begin{corollary}
\label{cor: ln a is strictly increasing}
$\ln a$ is strictly increasing in $a > 0$.
\end{corollary}

\begin{theorem}
\label{thm: derivative of log_a x}
For $a \in \R^+ \setminus \{1\}$ the logarithmic function $x \mapsto \log_a x$ is differentiable for all $x > 0$, and $\displaystyle \frac{d}{dx} \log_a x = \frac{1}{x \ln a}$.
\end{theorem}
\begin{proof}
Let $a > 0$, $a \neq 1$. By Theorem \ref{thm: a^x is continuously differentiable everywhere} and Theorem \ref{thm: basic properties of a^x} part (vii) the derivative of $a^x$ with respect to $x$ exists everywhere and is continuous and nonzero. Therefore for $x > 0$, by the Inverse Function Theorem and equation \eqref{eqn: the derivative of a^x is ln a times a^x} $\log_a x$ is continuously differentiable and
\[
\frac{d}{dx} \log_a x = \frac{1}{\left . \frac{d}{dy} a^y \right |_{y=\log_a x}} = \frac{1}{\left . (\ln a) a^y \right |_{y=\log_a x}} = \frac{1}{(\ln a) a^{\log_a x}} = \frac{1}{x \ln a}.
\]
\end{proof}

In particular by Theorem \ref{thm: our ln is the logarithm to base e}, Theorem \ref{thm: derivative of log_a x} and Theorem \ref{thm: the existence of e} we have the following.
\begin{corollary}
\label{cor: derivative of ln x}
For $x > 0$, $\displaystyle \frac{d}{dx} \ln x = \frac{1}{x}$.
\end{corollary}

The following integral is used in ``Late Transcendental Functions" editions of Calculus textbooks as the definition of $\ln x$.
\begin{theorem}
\label{thm: area under graph of 1/x}
For $x > 0$, $\displaystyle \int_1^x \frac{1}{t} \, dt = \ln x$.
\end{theorem}
\begin{proof}
By the first part of the Fundamental Theorem of Calculus, Corollary \ref{cor: derivative of ln x}, Theorem \ref{thm: basic properties of a^x} part (vi) and Theorem \ref{thm: our ln is the logarithm to base e}.
\end{proof}

The following limit is sometimes used as the definition of $e^x$, in which case the existence of the limit must be proved by means different from ours. \footnote{For instance by the Monotone Convergence Theorem, see \cite{Bartle_and_Sherbert_book} for the case $x = 1$.}
\begin{theorem}
\label{thm: sequence limit formula for e^x}
$\displaystyle \lim_{n \to \infty} \left ( 1 + \frac{x}{n} \right )^n = e^x$.
\end{theorem}
\begin{proof}
Let $x, t \in \R^+$.
By Theorem \ref{thm: basic properties of a^x} part (ii) and Theorem  \ref{thm: our ln is the logarithm to base e} we have
\begin{equation}
\label{eqn: exponential formula for (1 + x/t)^t}
\left (1 + \frac{x}{t} \right )^t = e^{t \log_e (1 + \frac{x}{t})} = e^{t \ln (1 + \frac{x}{t})}.
\end{equation}
The limit, as $t$ tends to infinity, of the left-hand side of the equation below has the indeterminate form $\infty \cdot 0$, so we write it as a fraction and apply L'Hopital's rule:
\begin{equation}
\label{eqn: application of L'Hopital's Rule}
\lim_{t \to \infty} t \ln \left (1 + \frac{x}{t} \right ) = \lim_{t \to \infty} \frac{\ln \left (1 + \frac{x}{t} \right )}{\frac{1}{t}} = \lim_{t \to \infty} \frac{-\frac{x}{t^2}}{-\frac{1}{t^2}(1 + \frac{x}{t})} = \lim_{t \to \infty} \frac{x}{1 + \frac{x}{t}} = x.
\end{equation}
By Theorem \ref{thm: basic properties of a^x} part (iv) and equations \eqref{eqn: exponential formula for (1 + x/t)^t} and \eqref{eqn: application of L'Hopital's Rule} we have $\lim_{t \to \infty} \left (1 + \frac{x}{t} \right )^t = e^x$.
The result follows on replacing $t > 0$ by $n \in \N$.
\end{proof}

The following Taylor-Maclaurin series is often used as the definition of $e^x$, in which case the convergence of the series for all $x \in \R$ must be proved differently than it is proved here (for instance by the Ratio Test.)
\begin{theorem}
\label{thm: Taylor-Maclaurin series for e^x}
For all $x \in \R$
\[
e^x = \sum_{k=0}^{\infty} \frac{1}{k!} x^k.
\]
\end{theorem}
\begin{proof}
By Taylor's theorem with the Lagrange form for the remainder and Theorem \ref{thm: the existence of e} and Corollary \ref{cor: e^x is its own derivative}, for all $x \in \R$ and all $n \in \N^*$
\[
e^x = \sum_{k=0}^n \frac{\frac{d^k}{dx^k} e^x |_{x=0}}{k!} x^k + \frac{\frac{d^{n+1}}{dx^{n+1}} e^x |_{x = \xi}}{(n+1)!}x^{n+1} = \sum_{k=0}^n \frac{1}{k!} x^k + \frac{e^{\xi}}{(n+1)!}x^{n+1}
\]
for some $\xi$ between 0 and $x$. But
\[
\left |\frac{e^{\xi}}{(n+1)!}x^{n+1} \right | \leq \max(1, e^x)\frac{|x|^{n+1}}{(n+1)!}
\]
and this bound on the remainder term tends to zero as $n$ tends to infinity. Indeed, if $N \in \N^*$ is such that $N + 1 > |x|$, then for $n > N$
\[
\frac{|x|^n}{n!} = \frac{|x|^N}{N!} \frac{|x|^{n-N}}{(N+1) \cdots n} \leq \frac{|x|^N}{N!} \frac{|x|}{n} \xrightarrow{n \to \infty} 0.
\]
\end{proof}

\vfill\eject

\end{document}